\documentclass[a4paper,12pt]{article}
\usepackage{amssymb,amsmath,amsthm,times}

\newtheorem{thm}{Theorem}[section]

\newtheorem{lem}[thm]{Lemma}

\newtheorem{thmabc}{Theorem}

\title{A Tauberian theorem for  Ingham summation method}
\author{Vytas Zacharovas\\
    Institute of Statistical Science\\
    Academia Sinica\\
    Taipei 115\\
    Taiwan}
\begin{document}

\maketitle

\begin{abstract}
\noindent The aim of this work is to prove a Tauberian theorem for
the Ingham summability method. The Tauberian theorem we prove is
then applied to analyze asymptotics of mean values of multiplicative
functions on natural numbers.

\end{abstract}

\section{Introduction}
Many problems in number theory involve estimating mean values
\begin{equation}
\label{F_mean_value} \frac{1}{n}\sum_{m=1}^nf(m)
\end{equation}
of
some complex valued function  $f:\mathbb{N}\to \mathbb{C}$.
In many cases $f(m)$ can be naturally represented as a sum
$\sum_{k|m}a_k$ where $a_k\in \mathbb C$. M\"{o}bius inversion
formula guarantees that for a given $f(m)$ such $a_k$ always exist
and are unique.
 Replacing $f(m)$ by $\sum_{k|m}a_k$ in the  sum of values of $f(m)$, we get
\[
\sum_{ m=1}^nf(m)=\sum_{m=1}^n\sum_{k|m}a_k
=\sum_{k=1}^na_k\left[\frac{n}{k} \right],
\]
here \([x]\) denotes the integer part of a real number \(x\).
Suppose we  want to know under which conditions the sequence of the
mean values (\ref{F_mean_value}) of $f(m)$ has a limit as $n\to
\infty$. This is equivalent to the question, under which conditions
on $a_k$ the sequence
\begin{equation}
\label{F_a_k} \frac{1}{n}\sum_{ k=1}^na_k\left[\frac{n}{k} \right]
\end{equation}
has a limit as $n\to \infty$. If, say
\[
\sum_{k\geqslant 1}\frac{|a_k|}{k}<\infty,
\]
then the theorem of Wintner (see e. g.
\cite{postnikov_1988_analytic}) states that
\begin{equation}
\label{limit}
\lim_{n\to\infty}\frac{1}{n}\sum_{
k=1}^na_k\left[\frac{n}{k} \right]=\sum_{k=1}^\infty\frac{a_k}{k}.
\end{equation}
It was shown in \cite{hardy_divergent}
that the convergence of the series
\begin{equation}
\label{the_series}
\sum_{k=1}^\infty\frac{a_k}{k}
\end{equation}
alone, does not necessarily imply the existence of the limit of sum
(\ref{F_a_k}) as $n\to \infty$. In 1910 Axer \cite{axer_1910} (see
also Chapter 3.6 of \cite{postnikov_1988_analytic}) proved that if
in addition to convergence of the series (\ref{the_series}) the
condition
\begin{equation}
\label{axer's_condition}
\sum_{k=0}^n|a_k|=O(n)
\end{equation}
is satisfied, then  the limit (\ref{limit}) exists.

We will show that in determining whether the sum (\ref{F_a_k}) has a
limit, an important role is played by the quantity
\begin{equation}
\label{define_S}
S(x)=\sum_{m\leqslant x}\sum_{k|m}a_k\log k
=\sum_{k\leqslant x}a_k\left[\frac{x}{k} \right]\log k.
\end{equation}
We will prove  (see Lemma \ref{convergence_of_the_series_lemma})
that condition $S(x)=O(x\log x)$ as $x\to \infty$ is enough to
ensure that the Dirichlet series
$\sum_{m=1}^\infty{a_m}{m^{-\sigma}}$ converges for all $\sigma>1$.
Which means that the function
\begin{equation}
\label{def_g} g(\sigma)=\sum_{m=1}^\infty\frac{a_m}{m^\sigma}
\end{equation}
will be correctly defined  in the infinite interval $\sigma>1$. The
next theorem shows that if $S(x)=o(n\log n)$ then the value of the
sum (\ref{F_a_k}) can be approximated by the values Dirichlet series
$g(\sigma)$ with $\sigma=1+\log^{-1} n$.
\begin{thm}
\label{first_thm} Suppose $a_n$ is a sequence of complex numbers
such that
\[
S(n)=\sum_{k\leqslant n}a_k\left[\frac{n}{k} \right]\log k=o(n\log
n).
\]
Then
\[
\frac{1}{n}\sum_{k\leqslant n}a_k\left[\frac{n}{k}
\right]=g\left(1+\frac{1}{\log n}\right)+o(1),
\]
as $n\to \infty$.
\end{thm}
The estimate of the above theorem will allow us to prove necessary
and sufficient conditions for existence of limit of  sum
(\ref{F_a_k}).
\begin{thm}
\label{Ttauberian_dirichlet} Suppose $a_m$ is a fixed sequence of
complex numbers. Then the limit
\[
\lim_{n\to \infty}\frac{1}{n}\sum_{k\leqslant n}a_k\left[\frac{n}{k}
\right]= C
\]
exists if and only if the following two conditions are satisfied
\begin{enumerate}
\item
\label{cond_a_m}
\begin{equation*}
\sum_{k\leqslant n}a_k\left[\frac{n}{k} \right]\log k=o(n\log
n),\quad\hbox{as}\quad n\to \infty
\end{equation*}
\item
\label{cond_series}
\[
\lim_{\sigma \searrow
1}\sum_{m=1}^\infty\frac{a_m}{m^\sigma}=C
\]
\end{enumerate}
Note that if condition \ref{cond_a_m} is satisfied then  the
infinite series in the formulation of condition \ref{cond_series}
converges for all $\sigma>1$.
\end{thm}

The last theorem is a direct analog of the  very first Tauberian
theorem that was proved by Tauber  in 1897.
\begin{thmabc}[Tauber, \cite{tauber_1897}]
A series
\begin{equation}
\label{series_a_n} \sum_{n=0}^\infty a_n
\end{equation}
converges and its sum  is equal to $A$, if and only if
\begin{equation}
\label{tauber's_sum} \sum_{k=0}^n ka_k=o(n)
\end{equation}
and  exists the limit
\[
\lim_{x\nearrow 1} \sum_{n=0}^\infty a_nx^n=A.
\]
\end{thmabc}
It can be  shown that Tauber's condition (\ref{tauber's_sum})
imposed on coefficients $a_j$ of the formal series
(\ref{series_a_n}) alone is enough to provide an asymptotic estimate
for partial sums
\[
\sum_{k=0}^n a_k=\phi(e^{-1/n})+o(1),
\]
where $\phi(z)=\sum_{j=0}^\infty a_jz^j$. 
Which is similar to the asymptotic given in the formulation of Theorem \ref{first_thm}.

It is but natural to ask how really useful are the stated theorems
for analyzing the mean values of concrete arithmetical functions?
Condition $\lim_{x \downarrow 1}g(x)=C$ does not cause any problem
if say the Dirichlet series $g(s)$ has a closed form expression
which allows us to obtain information on behavior of $g(x)$ for real
values of $x>1$ which are close to $1$. At a first glance the
condition $S(n)=o(n\log n)$ looks quite artificial and not much
easier to check than to prove that
$A(n)=\sum_{k=1}^na_k\left[\frac{n}{k} \right]=Cn+o(n)$, since
$S(n)$ is obtained just by replacing $a_k$ by  $a_k\log k$ in the
expression of $A(n)$. However, this condition is quite natural for a
wide class of sequences $a_m$ such that $f(m)$ defined as
$f(m)=\sum_{d|m}a_d$ is a completely multiplicative function of $m$,
that is a function satisfying equation
\begin{equation}
\label{multiplicativeness} f(mn)=f(m)f(n),
\end{equation}
for any $m,n\in\mathbb{N}$. It is easy to check that if a completely
multiplicative function $f$ is bounded $|f(m)|\leqslant 1$, then the
condition $S(n)=o(n\log n)$ will be satisfied if
\begin{equation}
\label{cond_1} \sum_{p\leqslant n}\frac{|f(p)-1|}{p}\log p=o(\log
n),
\end{equation}
or
\begin{equation}
\label{cond_2} \sum_{m\leqslant n} \left| \sum_{p\leqslant
n/m}f(p)\log p -\frac{n}{m}\right| =o(n\log n),
\end{equation}
here and further we will follow the tradition to denote by
$\sum_{p}$ and $\prod_{p}$ the sums and products over prime numbers
$p$. This allows us to deduce a few classical results for the mean
values of multiplicative functions. For example, it can be shown
that if any of the above two conditions (\ref{cond_1}) or
(\ref{cond_2}) is satisfied for a fixed multiplicative function $f$
such that $|f(m)|\leqslant 1$, then  Theorem \ref{first_thm} implies
an estimate
\[
\frac{1}{n}\sum_{m=1}^n
f(m)=\prod_{p}\frac{1-\frac{1}{p^{1+{1}/{\log n}}}
}{1-\frac{f(p)}{p^{1+{1}/{\log n}}} }+o(1),
\]
as $n\to \infty$.

 Results with similar or even stronger error terms
than in the inequality of the next theorem can be proven by the
method of Hal\'{a}sz (see e.g. Chapter 19 of monograph
\cite{elliott_vol_2} and papers
\cite{halasz_1968},\cite{manstavicius_1979},\cite{maciulis_1988} and
\cite{levin_fainleib_1970}). We present its  proof bellow just to
demonstrate the connection between the Ingham summation method and
the mean values of multiplicative functions. Its proof is an easy
consequence of the the same estimates that enable us to prove
Theorem \ref{Ttauberian_dirichlet}.
\begin{thm}
\label{T_multiplicative} Suppose $f(m)$ -- completely multiplicative
function such that
$|f(m)|\leqslant 1$ then 
\[
\left|\frac{1}{n}\sum_{m=1}^nf(m)-\prod_{p\leqslant
n}\frac{1-\frac{1}{p}}{1-\frac{f(p)}{p}}\right|\leqslant
R(\alpha)\mu_n(\alpha),
\]
for any $\alpha>1$ , with $R(\alpha)$ -- a positive constant, which
depends on $\alpha$ only, and
\[
\mu_n(\alpha)=\left(\frac{1}{\log n}\sum_{p\leqslant
n}\frac{\bigl|f(p)-1\bigr|^\alpha}{p}\log p\right)^{1/\alpha}.
\]
\end{thm}
Similar result holds for general multiplicative functions, i. e.
such functions that condition (\ref{multiplicativeness}) is required
to be satisfied only for coprime pairs of natural numbers $m,n$. It
follows from our proof of Theorem \ref{T_multiplicative}  that its
modified version for general multiplicative functions will hold if
we weaken condition $|f(m)|\leqslant1$ to requirement  that
$|f(1)+f(2)+\cdots f(m)|\leqslant Dm$ for all $m\geqslant 1$, with
some fixed $D$.

Unfortunately our proof of Theorems \ref{first_thm} and \ref{Ttauberian_dirichlet}
 is not elementary since it  relies on the
estimate of the number of primes in short intervals (Theorem
\ref{number_of_primes_in_interval}) that has  originally been proved
(see e.g. \cite{karatsuba}) using a number of non-trivial facts
about distribution of zeroes of the Riemann Zeta function.

 The Tauberian theorem we prove can be reformulated in terms of
the theory of summation of divergent series. Recall that a formal
series $\sum_{m=1}^\infty c_m$ is called summable in the sense of
Ingham if there exists  a complex number $C$ such that
\[
\lim_{n\to
\infty}\sum_{m=1}^n\frac{m}{n}\left[\frac{n}{m}\right]c_m=C,
\]
in which case we write
\[
(I)\sum_{m=1}^\infty c_m=C.
\]
Suppose $0<\lambda_1<\lambda_2<\cdots\lambda_n<\cdots$ is  a
sequence of positive strictly increasing real numbers. Then we say
that a formal series $\sum_{m=0}^\infty c_m$ is $(A,\lambda_n)$
summable and its value is $C$ if
\[
\lim_{x\downarrow 0}\sum_{m=0}^\infty c_me^{-\lambda_m x}=C,
\]
in which case we write
\[
(A,\lambda_n)\sum_{m=1}^\infty c_m=C.
\]
 With these notations our tauberian theorem means that
$(I)\sum_{m=1}^\infty c_m=C$ if and only if
\[
\sum_{m=1}^n\frac{m}{n}\left[\frac{n}{m}\right]c_m\log m=o(\log
n),\quad \hbox{as}\quad n\to \infty
\]
and $(A,\log n)\sum_{m=1}^\infty c_m=C$.

The analogy between the classical Tauber's theorem and the theorem
we prove leads us to expect that a wide class of summability methods
is connected to some class of $(A, \lambda_n)$ summability methods
in such a way that a formal series $\sum_{m=0}^\infty c_m$ is
summable if and only if it is $(A, \lambda_n)$ summable and the
partial sums defining summability method with $\lambda_mc_m$ instead
of $c_m$ are $o(\lambda_n)$. We thus prove that the Ingham
summability method is connected in  this sense with $(A,\log n)$
method. It was shown in \cite{zakh_cesaro_2001} that this pattern
holds also for the Ces\`{a}ro summability methods $(C,\theta)$ with
$\theta>-1$ which are proved to be connected to $(A,n)$ method. In
the same paper we exploited the connection of Ces\`{a}ro summation
method with the multiplicative functions on permutations to obtain
 an analog of the Theorem \ref{T_multiplicative} providing the
asymptotic estimate of the mean value of the multiplicative function
on permutations.

\section{Proofs}
Let us start by introducing notations that will be used later in the
paper. We will denote by $\Psi(x)$ the Chebyshev's function
\[
\Psi(x)=\sum_{m\leqslant x}\Lambda(m),
\]
where $\Lambda(m)$ -- Mangoldt's function. We will also denote
\[
\Delta(x,y)=\Psi(y)-\Psi(x)-(y-x).
\]
Later we will need an upper bound estimate  of $\Delta(x,y)$ which
we formulate as the next theorem. In fact much stronger estimate is
known (see \cite{karatsuba}). However we formulate the weakest
estimate that we know to be sufficient for our proof of Theorem
\ref{Tmain}.
\begin{thm}[\cite{karatsuba}]
\label{number_of_primes_in_interval} Suppose $c>0$ is a fixed
constant. There exists a constant $\eta$ satisfying condition
$0<\eta<1$ such that
\[
\Delta(x,x+h)\ll \frac{h}{\log x}, \quad\hbox{when}\quad h\geqslant
c x^{\eta},
\]
for $x\geqslant 2$,  the constant in symbol $\ll$ is absolute,
depending only on $c$ and $\eta$.
\end{thm}
For any $t>0$ we define a positive multiplicative function
\[
f_t(m)=\sum_{d|m}\frac{\mu (d)}{d^t}=\prod_{p|m}\left(
1-\frac{1}{p^t} \right)>0,
\]
where $\mu(d)$ is M\"{o}bius function. The Dirichlet generating
series of $f_t(m)$ is
\begin{equation}
\label{gen_f_for_f_t}
L_t(s)=\sum_{m=1}^\infty \frac{f_t(m)}{m^s}
=\sum_{m=1}^\infty\frac{1}{m^s}\sum_{m=1}^\infty\frac{\mu
(d)}{d^{s+t}}= \frac{\zeta(s)}{\zeta(s+t)},
\end{equation}
where $\zeta(s)=\sum_{m=1}^\infty m^{-s}$ is the Riemann Zeta
function. We will denote the partial sums of $f_t(m)$ as
\[
F_t(x)=\sum_{1\leqslant m\leqslant x}f_t(m),\quad \mbox{for} \quad
x\geqslant 1.
\]
Later we will need the estimates of the various sums involving
\(f_t(m)\), which we formulate as the following lemma.
\begin{lem}
\label{Lestimf_m} For any $x>1$ and $t>0$ we have
\begin{equation}
\label{Lestimf_m_1} \sum_{m\leqslant x}\frac{f_t(m)}{m}\ll 1+t\log
x,
\end{equation}
\begin{equation}
\label{partial_sums_f_t} F_t(x)=\frac{x}{\zeta(1+t)}+O(x^{1-t})+
O\left(\sum_{d\leqslant x}\frac{1}{d^t}\right),
\end{equation}
\begin{equation}
\label{Lestimf_m_2} \sum_{m\leqslant
x}\frac{f_t(m)}{m}=\sum_{d\leqslant x}\frac{\mu(d)}{d^{1+t}}\log
\frac{x}{d}+O(1),
\end{equation}
\begin{equation}
\label{Lestimf_m_3} F_t(x)-F_t\left(\frac{x}{2} \right)\ll
x\left(\frac{1}{\log x}+t\right),
\end{equation}
for $k\geqslant 2$ we have
\begin{equation}
\label{Lestimf_m_4} \int_0^\infty F_t( x)\left( \frac{1
}{k^t}-\frac{1 }{(k+1)^t}\right) dt= x\int_0^\infty \left( \frac{1
}{k^t}-\frac{1
}{(k+1)^t}\right)\frac{dt}{\zeta(1+t)}+O\left(\frac{x}{k\log^2xk}\right).
\end{equation}

\end{lem}
\begin{proof}  The estimates of the lemma are
trivial if  $x\leqslant 3$, therefore throughout the proof we will
assume that $x>3$.  Recalling the formula for the Dirichlet
generating function (\ref{gen_f_for_f_t}) of $f_t(m)$ we obtain
\[
\sum_{m\leqslant x}\frac{f_t(m)}{m} \leqslant e \sum_{m\leqslant
x}\frac{f_t(m)}{m^{1+\frac{1}{\log x}}}\leqslant
e\frac{\zeta(1+\frac{1}{\log x})}{\zeta(1+\frac{1}{\log x}+t)}\ll
1+t\log x,
\]
since $\frac{1}{u-1}<\zeta(u)<\frac{u}{u-1}$ for any $u
>1$. This proves (\ref{Lestimf_m_1}).

To prove the next two estimates we replace $f_t(m)$ by a sum
$\sum_{d|m}\mu(d)d^{-t}$. This way we obtain
\[
\begin{split}
F_t(x)&=\sum_{m\leqslant x}f_t(m)=\sum_{m\leqslant
x}\sum_{d|m}\frac{\mu(d)}{d^t}=\sum_{d\leqslant
x}\frac{\mu(d)}{d^t}\left[\frac{x}{d}\right] 
\\
&=\frac{x}{\zeta(1+t)}+O(x^{1-t})+ O\left(\sum_{d\leqslant
x}\frac{1}{d^t}\right),
\end{split}
\]
here we estimated $\sum_{d>x}\mu(d)d^{-1-t}\ll x^{-t}$ by applying
partial summation and utilizing the well-known fact that
\begin{equation}
\label{sum_of_mu} \Bigl|\sum_{d\leqslant
m}\frac{\mu(d)}{d}\Bigr|\leqslant 1
\end{equation} for
all $m\geqslant 1$. This proves the estimate
(\ref{partial_sums_f_t}). In a similar way
\begin{equation*}
\begin{split}
\sum_{m\leqslant x}\frac{f_t(m)}{m}&=\sum_{m\leqslant
x}\frac{1}{m}\sum_{d|m}\frac{\mu(d)}{d^t} 
=\sum_{d\leqslant x}\frac{\mu(d)}{d^{1+t}}\sum_{k\leqslant
x/d}\frac{1}{k}
\\
&=\sum_{d\leqslant
x}\frac{\mu(d)}{d^{1+t}}\left(\log\frac{x}{d}-\gamma
+O\left(\frac{d}{x}\right)\right)=\sum_{d\leqslant
x}\frac{\mu(d)}{d^{1+t}}\log\frac{x}{d}+O(1),
\end{split}
\end{equation*}
here we have used the estimate (\ref{sum_of_mu}) of partial sums of
$\mu(d)/d$. The estimate (\ref{Lestimf_m_2}) is proved.

 Differentiating by $s$ the
Dirichlet series of $f_t(m)$ we get
\begin{equation*}
\begin{split}
\sum_{m=1}^\infty \frac{f_t(m)\log
m}{m^s}&=-\frac{d}{ds}\frac{\zeta(s)}{\zeta(s+t)}=-
\frac{\zeta(s)}{\zeta(s+t)}\left(\frac{\zeta'(s)}{\zeta(s)}-\frac{\zeta'(s+t)}{\zeta(s+t)}
\right)
\\
&=\sum_{m=1}^\infty \frac{f_t(m)}{m^s}\sum_{m=1}^\infty
\frac{\Lambda(m)}{m^s} \left( 1-\frac{1}{m^t} \right).
\end{split}
\end{equation*}
Equating the coefficients of $\frac{1}{m^s}$ in the above expression
and summing by $m$ such that $m\leqslant x$ we get an identity
\[\sum_{m\leqslant x}f_t(m)\log
m=\sum_{k\ell\leqslant x}f_t(k)\Lambda(\ell) \left(
1-\frac{1}{\ell^t} \right),
\]
therefore
\begin {equation*}
\begin{split}
F_t(x)-F_t\left(\frac{x}{2} \right)&\leqslant \frac{1}{\log
\frac{x}{2}}\sum_{m\leqslant x}f_t(m)\log m\leqslant \frac{1}{\log
\frac{x}{2}}\sum_{d\ell\leqslant x}f_t(d)\Lambda(\ell)\left(
1-\frac{1}{\ell^t} \right)
\\
&\leqslant \frac{1}{\log \frac{x}{2}}\sum_{d\leqslant
x}f_t(d)\Psi\left(\frac{x}{d}\right)
\ll \frac{x}{\log \frac{x}{2}}\sum_{d\leqslant x}\frac{f_t(d)}{d}
\\
&\ll x\left(\frac{1}{\log x}+t\right),
\end{split}
\end{equation*}
for $x\geqslant 3$. Here we have used the fact that $\Psi(x)=O(x)$
and applied the already proven estimate (\ref{Lestimf_m_1}). This
proves (\ref{Lestimf_m_3}).

Applying the identity $f_t(m)=\sum_{d|m}\mu(d)d^{-t}$ we obtain
\begin{equation*}
\begin{split}
&\int_0^\infty F_t( x)\left( \frac{1 }{k^t}-\frac{1
}{(k+1)^t}\right) dt =\sum_{m\leqslant x} \int_0^\infty \left(
\frac{1 }{k^t}-\frac{1
}{(k+1)^t}\right)\Bigl(\sum_{d|m}\frac{\mu(d)}{d^t} \Bigr)\,dt
\\
&\quad
=\sum_{d\leqslant x}\mu(d)\left[ \frac{x}{d} \right] \int_0^\infty
\left( \frac{1 }{k^t}-\frac{1 }{(k+1)^t}\right)  \frac{dt}{d^t}
\\
&\quad=x\sum_{d\leqslant x} \frac{\mu(d)}{d} \int_0^\infty \left(
\frac{1 }{k^t}-\frac{1 }{(k+1)^t}\right)  \frac{dt}{d^t}
+O\left(\sum_{d\leqslant x}\frac{\log\left( 1+\frac{1}{k} \right)
}{\log^2dk}\right),
\end{split}
\end{equation*}
for all \(x\geqslant 1\). Using the estimate (\ref{sum_of_mu}) of
sums of $\mu(d)/d$ and applying partial summation
 we can estimate the tail of the series in the last expression as
\[
\left|\sum_{d> x} \frac{\mu(d)}{d} \int_0^\infty \left( \frac{1
}{k^t}-\frac{1 }{(k+1)^t}\right)  \frac{dt}{d^t}\right| \leqslant 2
\int_0^\infty \left( \frac{1 }{k^t}-\frac{1 }{(k+1)^t}\right)
\frac{dt}{x^t} \leqslant \frac{2}{k\log^2kx}.
\]
Evaluating the sum inside the symbol $O(\ldots)$ in the previous
estimate by means of inequality $\sum_{1\leqslant d\leqslant
x}\frac{1}{\log^2dk}\ll \frac{x}{\log^2xk}$ we complete the proof of
the estimate (\ref{Lestimf_m_4}).

\end{proof}
\begin{lem}
\label{convergence_of_the_series_lemma} Suppose sequence $a_k$ is
such that for any $v>1$
\begin{equation}
\label{S_k_bounded} \lim_{k\to \infty} \frac{|S(k)|}{k^{v}}=0,
\end{equation}
then series \(\sum_{m=1}^\infty\frac{a_m}{m^v}\) converges for all
\(v>1\).
\end{lem}
\begin{proof} Condition of the lemma implies by summation by parts that Dirichlet series
\begin{equation}
\label{difference_of_s_m} \sum_{m=1}^\infty\frac{S(m)-S(m-1)}{m^s}
\end{equation}
converges for all $s>1$. Recalling the definition (\ref{define_S})
of $S(m)$ we can express the difference $S(m)-S(m-1)$ as a sum of
$a_k\log k$ in the following way
$$
S(m)-S(m-1)=\sum_{k|m}a_k\log k,\quad \mbox{for}\quad m\geqslant 1.
$$

 This means that if we multiply our convergent series (\ref{difference_of_s_m}) by an
absolutely convergent series
$\sum_{m=1}^\infty\frac{\mu(m)}{m^s}=1/\zeta(s)$ then the resulting
series
\[
\sum_{m=1}^\infty\frac{a_m\log m}{m^s}
\]
is also convergent for all $s>1$. This in its turn implies that if
we integrate the above series with respect to $s$,  then the
resulting the series
\[
\sum_{m=1}^\infty\frac{a_m}{m^s}
\]
is also convergent for all $s>1$.
\end{proof}
\begin{lem} Suppose sequence $a_k$ is such that
\begin{equation}
\label{S_k_bounded} \lim_{k\to \infty} \frac{|S(k)|}{k^{v}}=0,
\end{equation}
for any $v>1$, then by Lemma \ref{convergence_of_the_series_lemma}
the function $g(s)=\sum_{m=1}^\infty \frac{a_m}{m^s}$ will be
correctly defined for all $s>1$ and the identity
\begin{equation}
\label{formula_for_difference}
\begin{split}
&\sum_{m=1}^na_m\left[\frac{n}{m} \right] -ng\left(1+\frac{1}{\log
n} \right)-\frac{S(n)}{\log n}
\\
&\quad=\sum_{k=2}^{n-1}S(k) \int_0^\infty\left( \frac{F_t\left(
\frac{n}{k}\right) }{k^t}-\frac{F_t\bigl( \frac{n}{k+1}\bigr)
}{(k+1)^t}\right) dt -n\sum_{k=2}^\infty
S(k)\int_{\sigma}^\infty\left(  \frac{1}{k^u} -\frac{1}{(k+1)^u}
\right) \frac{du}{\zeta (u)},
\end{split}
\end{equation}
holds for all $n\geqslant 2$.  Here we assume that $\sum_{k=2}^1(\ldots)=0$.
\end{lem}
\begin{proof}
The M\"{o}bius inversion formula yields
$$
a_m\log m=\sum_{k|m}\mu \left(\frac{m}{k}\right)
\bigl(S(k)-S(k-1)\bigr) \quad \mbox{when}\quad m\geqslant 1.
$$

 Inserting  the above expression for $a_k$ into the righthand
side of the identity (\ref{formula_for_difference}) in the statement
of our theorem,
 denoting
\[
\sigma=1+\frac{1}{\log n}
\]
and taking into account that $S(1)=S(0)=0$ we obtain
\begin{equation*}
\begin{split}
\sum_{m=1}^na_k\left[\frac{n}{m} \right] -ng(\sigma)&=
\sum_{m=2}^na_k\left[\frac{n}{m} \right]
-n\sum_{m=2}^\infty\frac{a_m}{m^\sigma}
\\
&=\sum_{m=2}^n\left[\frac{n}{m} \right]\frac{1}{\log m}\sum_{k|m}\mu
\left(\frac{m}{k}\right) \bigl(S(k)-S(k-1)\bigr)
\\
&\quad-n\sum_{m=2}^\infty \frac{1}{m^\sigma \log m}\sum_{k|m}\mu
\left(\frac{m}{k}\right) \bigl(S(k)-S(k-1)\bigr)
\\
\end{split}
\end{equation*}
Changing the order of summation  of the two sums occurring in last expression
we obtain
\begin{equation}
\label{Eq_diff}
\begin{split}
\sum_{m=1}^na_k\left[\frac{n}{m} \right] -ng(\sigma)&=\sum_{k=2}^n\bigl( S(k)-S(k-1) \bigr) \sum_{\scriptstyle
m:\,1\leqslant m\leqslant n \atop \scriptstyle k|m }\left[
\frac{n}{m} \right]\frac{\mu \left(\frac{m}{k}\right)}{\log m}
\\
&\quad-n\sum_{k=2}^\infty \bigl( S(k)-S(k-1) \bigr) \sum_{ m:\, k|m
}\frac{\mu \left(\frac{m}{k}\right)}{m^\sigma \log m},
\end{split}
\end{equation}
for $n\geqslant 2$. Let us show that the condition
(\ref{S_k_bounded}) imposed upon $|S(k)|$ guarantees that the
exchanging of the order of summation is justified. Indeed, Lemma
\ref{convergence_of_the_series_lemma} guarantees the convergence of
the series $\sum_{m=1}^\infty \frac{a_m}{m^\sigma}$, which means
that
\[
\sum_{m=2}^\infty \frac{a_m}{m^\sigma }
=\lim_{N\to \infty} \sum_{m=2}^N \frac{1}{m^\sigma \log
m}\sum_{k|m}\mu \left(\frac{m}{k}\right) \bigl(S(k)-S(k-1)\bigr)
\]
For any finite $N$ we can exchange the order of summation in the
 expression under the limit sign and fixing an integer $M\geqslant 3$ we obtain
\begin{equation}
\label{change_of_summation}
\begin{split}
\sum_{m=2}^\infty \frac{a_m}{m^\sigma }
&=\lim_{N\to \infty} \sum_{k=2}^N \frac{S(k)-S(k-1)}{k^\sigma}
\sum_{\ell\leqslant N/k} \frac{\mu(\ell)}{\ell^\sigma \log (k\ell)}
\\
&= \sum_{k=2}^{M-1} \frac{S(k)-S(k-1)}{k^\sigma} \sum_{\ell=1}^\infty
\frac{\mu(\ell)}{\ell^\sigma \log (k\ell)}
\\&\quad+\lim_{N\to \infty} \sum_{k=M}^N \frac{S(k)-S(k-1)}{k^\sigma}
\sum_{\ell\leqslant N/k} \frac{\mu(\ell)}{\ell^\sigma \log (k\ell)}
\\
&= \sum_{k=2}^{M-1} \frac{S(k)-S(k-1)}{k^\sigma} \sum_{\ell=1}^\infty
\frac{\mu(\ell)}{\ell^\sigma \log
(k\ell)}+O\left(\frac{1}{M^{\sigma-\sigma'}}\right),
\end{split}
\end{equation}
where $\sigma'$ is a fixed number such that $1<\sigma'<\sigma$.
Indeed
\[
\sum_{k=M}^N \bigl(S(k)-S(k-1)\bigr)\frac{1}{k^{\sigma}}
\sum_{\ell\leqslant N/k} \frac{\mu(\ell)}{\ell^\sigma \log
(k\ell)}=\sum_{k=M}^N \bigl(S(k)-S(k-1)\bigr)\alpha_k
\]
with $\alpha_k=\frac{1}{k^{\sigma}} \sum_{\ell\leqslant N/k}
\frac{\mu(\ell)}{\ell^\sigma \log (k\ell)}$, which are such that
$\alpha_k\ll 1/k^{\sigma}$ and
\[
|\alpha_k-\alpha_{k+1}|\ll
\frac{1}{k^{\sigma+1}}+\frac{1}{N^\sigma}\left(\left[\frac{N}{k}\right]-\left[\frac{N}{k+1}\right]\right).
\]
By condition of our lemma $S(n)\ll n^{\sigma'}$. This by means of
summation by parts and applying the above upper bound for
$|\alpha_k-\alpha_{k-1}|$ leads to estimate
\[
\begin{split}
\sum_{k=M}^N \bigl(S(k)-S(k-1)\bigr)\alpha_k &\ll
\frac{|S(M-1)|}{M^{\sigma}}+\frac{|S(N)|}{N^{\sigma}}
\\
&\quad+\sum_{k=M}^{N-1} |S(k)|\left(\frac{1}{k^{\sigma+1}}
+\frac{1}{N^\sigma}\left(\left[\frac{N}{k}\right]
-\left[\frac{N}{k+1}\right]\right)\right)
\\&\ll\frac{1}{M^{\sigma-\sigma'}}+\frac{1}{N^{\sigma-\sigma'}}
\end{split}
\]
whence we conclude that the upper limit of the above expression as
$N\to \infty$ does not exceed $O(M^{-(\sigma-\sigma')})$. This
proves (\ref{change_of_summation}). Letting $M\to \infty$ in
(\ref{change_of_summation}) we conclude that the change of summation
in (\ref{Eq_diff}) is justified.

Let us express the quantities involving $\mu(d)$ in the identity
(\ref{Eq_diff}) in terms of the function $f_t(m)$
\begin{equation*}
\begin{split}
 \sum_{\scriptstyle 1\leqslant m\leqslant n \atop \scriptstyle k|m }\left[ \frac{n}{m} \right]\frac{\mu \left(\frac{m}{k}\right)}{\log m}
 &=\sum_{1\leqslant d \leqslant \frac{n}{k}}
\left[ \frac{n}{kd} \right]\frac{\mu (d)}{\log kd}= \sum_{1\leqslant
m \leqslant \frac{n}{k}}\sum_{d|m}\frac{\mu (d)}{\log kd}
\\
&=\sum_{1\leqslant m \leqslant
\frac{n}{k}}\sum_{d|m}\int_0^\infty\frac{\mu(d)}{(dk)^t}dt
=\sum_{1\leqslant m \leqslant
\frac{n}{k}}\int_0^\infty\frac{1}{k^t}\prod_{p|m} \left(
1-\frac{1}{p^t} \right)\,dt
\\
&=\int_0^\infty\frac{F_t\left( \frac{n}{k}\right) }{k^t}\,dt.
\end{split}
\end{equation*}
In a similar fashion we obtain
$$
\sum_{  k|m }\frac{\mu \left(\frac{m}{k}\right)}{m^\sigma \log m}
=\sum_{d=1}^\infty \frac{\mu (d)}{k^\sigma d^\sigma \log
kd}=\int_{\sigma}^\infty \frac{du}{k^u\zeta (u)}.
$$
Inserting the above expressions into (\ref{Eq_diff}) and using
summation by parts in the resulting identities  we complete the
proof of the lemma.
\end{proof}

The estimate provided by the following theorem is a crucial part of
our argument that will enable us to obtain the results stated in the
introduction.
\begin{thm}
\label{Tmain}  Suppose sequence $a_k$ is such that for any $v>1$
\begin{equation}
\label{S_k_bounded} \lim_{k\to \infty} \frac{|S(k)|}{k^{v}}=0,
\end{equation}
then the function \(g(v)=\sum_{m=1}^\infty\frac{a_m}{m^v}\) is
correctly defined for all $v>1$  and for \(n\geqslant 2\) we have
\begin {equation}
\label{main_inequality}
\begin{split}
&\left|\sum_{m=1}^na_k\left[\frac{n}{m} \right] -ng\left(
1+\frac{1}{\log n} \right) 
\right|
\ll \sum_{k=2}^{n}c_{n,k}|S(k)| +\frac{n}{\log
n}\sum_{k=n}^\infty\frac{|S(k)|}{k^{2 +1/\log n}\log k}
,
\end{split}
\end{equation}
where $c_{n,k}$ are non-negative real constants that satisfy the
condition
\begin{equation}
\label{cond_c_m} \sum_{k=2}^{n-1}c_{n,k}k(\log
k)^{\varepsilon}\leqslant C(\varepsilon)n(\log n)^{\varepsilon -1},
\end{equation}
for any $0<\varepsilon \leqslant 1$, where $C(\varepsilon)>0$  is a
constant which depends on $\varepsilon$ only. Moreover
\begin{equation}
\label{small_k}
c_{n,k}=o(n), \quad \hbox{as}\quad n\to \infty
\end{equation}
for any fixed $k$.
\end{thm}

\begin{proof} Let us denote
\begin{equation}
\label{R_n} R_n=\sum_{m=1}^na_k\left[\frac{n}{m} \right]
-ng\left(1+\frac{1}{\log n} \right)-\frac{S(n)}{\log n}
\end{equation}
We will prove the theorem by estimating the quantities involved in
the right hand side of identity (\ref{formula_for_difference})
expressing $R_n$ in terms of quantities involving sums of $f_t(m)$.
Throughout the proof we will denote
$$
\sigma=1+\frac{1}{\log n}.
$$
Applying inequality \mbox{$\zeta(u)>\frac{1}{u-1}$},  which is true
for all $u>1$, we obtain
\begin{equation}
\label{int_1_sigma}
\begin{split}
\int_{1}^{\sigma}\left(  \frac{1}{k^u} -\frac{1}{(k+1)^u} \right)
\frac{du}{\zeta (u)}&<\int_1^\sigma \frac{u-1}{k^u}\left( 1-e^{-u\log \left( 1+\frac{1}{k}\right)} \right) du
\\
&<\frac{\sigma}{k^2}\int_{0}^{\sigma-1}u\,du=
\frac{\sigma}{2k^2\log^2n}.
\end{split}
\end{equation}
For $k\geqslant n$ we have
\begin{equation}
\label{int_sigma_infty}
\begin{split}
\int_{\sigma}^\infty\left(  \frac{1}{k^u} -\frac{1}{(k+1)^u} \right)
\frac{du}{\zeta (u)}
\ll
\frac{1}{k^{\sigma +1}\log n\log k}.
\end{split}
\end{equation}
Putting $x=\frac{n}{k+1}$ in (\ref{Lestimf_m_4}) we obtain
\[
 \int_0^\infty F_t\left( \frac{n}{k+1}\right)\left( \frac{1
}{k^t}-\frac{1 }{(k+1)^t}\right) dt= \frac{n}{k+1}\int_0^\infty
\left( \frac{1 }{k^t}-\frac{1
}{(k+1)^t}\right)\frac{dt}{\zeta(1+t)}+O\left(\frac{n}{k^2\log^2n}\right).
\]
Let us now use the above estimate together with (\ref{int_1_sigma})
and (\ref{int_sigma_infty}) to further simplify  the expression of
$R_n$.
\begin{multline*}
R_n=\sum_{k=2}^{n-1}S(k)\left[
\int_0^\infty\frac{ F_t\left( \frac{n}{k}\right) -F_t\bigl( \frac{n}{k+1}\bigr) }{k^t} dt  -\frac{n}{k(k+1)}\int_{0}^\infty
\frac{dt}{k^t\zeta (1+t)}        \right]
\\
+O\left(\frac{n}{\log n}\sum_{k=n}^\infty\frac{|S(k)|}{k^{\sigma
+1}\log k}+ \frac{n}{\log^2 n}\sum_{k=2}^n\frac{|S(k)|}{k^2}\right).
\end{multline*}
Suppose $\sqrt{n}\leqslant k \leqslant n-1$, then
$\frac{n}{k}-\frac{n}{k+1}=\frac{n}{k(k+1)}<1$. This means that
there can be only one natural number between $\frac{n}{k}$ and
$\frac{n}{k+1}$. In which case, if there exists
 such $m$ that $\frac{n}{k}\geqslant m > \frac{n}{k+1}$, we have
 $k\leqslant \frac{n}{m}$ and $k+1>\frac{n}{m}$. This means that
 $\left[ \frac{n}{m}\right] \geqslant k > \left[ \frac{n}{m}\right] -1$.
Which implies that $k=\left[ \frac{n}{m}\right]$. And conversely,
for $k=\left[ \frac{n}{m}\right]$, we have
 $\frac{n}{k}\geqslant m > \frac{n}{k+1}$. Thus the only natural
 numbers \(k\) in the interval \(\sqrt{n}\leqslant k \leqslant n-1\)
 such that the interval \(\left[ \frac{n}{k},\frac{n}{k+1}\right)\)
 contains some natural number \(m\) and subsequently \(F_t\left( \frac{n}{k}\right)
-F_t\bigl( \frac{n}{k+1}\bigr)=f_t(m)\) are of the form \(k=[n/m]\).
This observation allows us to further simplify the estimate of the
sum over $k>\sqrt{n}$ in the estimate of $R_n$ and obtain
\begin{equation*}
\begin{split}
|R_n| &\leqslant \sum_{2\leqslant k<\sqrt{n}}|S(k)|\left|
\int_0^\infty\frac{ F_t\left( \frac{n}{k}\right) -F_t\bigl(
\frac{n}{k+1}\bigr) }{k^t} \,dt  -\frac{n}{k(k+1)}\int_{0}^\infty
\frac{dt}{k^t\zeta (1+t)}        \right|
\\
&\quad +
\sum_{2\leqslant m\leqslant \sqrt{n}}\left|S\left( \left[\frac{n}{m} \right] \right)\right|\int_0^\infty
\frac{f_t(m)}{[n/m]^t}\,dt
+O\left(\frac{n}{\log
n}\sum_{k=n}^\infty\frac{|S(k)|}{k^{\sigma +1}\log k}+
\frac{n}{\log^2 n}\sum_{k=2}^n\frac{|S(k)|}{k^2}\right).
\end{split}
\end{equation*}
Thus the inequality (\ref{main_inequality}) holds if for $k\leqslant
\sqrt{n}$ we put
\begin{equation}
\label{define_c_n_k}
c_{n,k}=\left| \int_0^\infty\frac{ F_t\left(
\frac{n}{k}\right) -F_t\bigl( \frac{n}{k+1}\bigr) }{k^t} dt
-\frac{n}{k(k+1)}\int_{0}^\infty \frac{dt}{k^t\zeta (1+t)}
\right|+\frac{n}{k^2\log^2 n}
\end{equation}
and for $k>\sqrt{n}$ define
\begin{equation}
\label{c_nk_k>sqrt}
c_{n,k}=
\begin{cases}
\frac{n}{k^2\log^2 n},& \text{if $\sqrt{n}<k\leqslant n-1$ and
$k\not=[n/m]$ for any $m\leqslant \sqrt{n}$},
\\
\int_0^\infty \frac{f_t(m)}{[n/m]^t}dt+\frac{n}{k^2\log^2 n},&
\text{if $\sqrt{n}<k\leqslant n-1$ and $k=[n/m]$ for some
$m\leqslant \sqrt{n}$},
\\
\frac{n}{\log n}, & \text{if $k=n$}.
\end{cases}
\end{equation}

Plugging the estimate (\ref{partial_sums_f_t}) of $F_t(x)$  into our
definition  of $c_{n,k}$ in (\ref{define_c_n_k}) after some easy
calculations we conclude that for fixed $k$ we have $c_{n,k}=o(n)$.

It remains to check that thus defined $c_{n,k}$ satisfy the
condition (\ref{cond_c_m}) for any fixed $0<\varepsilon\leqslant 1$.
We will do this by splitting the sum involving $c_{n,k}$ into three
parts
\begin{equation}
\label{estimate_of_K}
\begin{split}
\sum_{2\leqslant k\leqslant n-1}c_{n,k}k(\log k)^{\varepsilon}&=
\sum_{k\leqslant n^\alpha}c_{n,k}k(\log k)^{\varepsilon}+
\sum_{n^\alpha< k< \sqrt{n}}c_{n,k}k(\log k)^{\varepsilon}+
\sum_{\sqrt{n}\leqslant k\leqslant n-1}c_{n,k}k(\log
k)^{\varepsilon}
\\
 &=:K_1+K_2+K_3.
\end{split}
\end{equation}
Here and further $0<\alpha<1/2$ will be fixed arbitrarily chosen
number, upon which we will later impose additional upper bound
conditions.

 The case of estimating $K_3$ the sum of $c_{n,k}$ over
interval $\sqrt{n}\leqslant k\leqslant n-1$ is the easiest. By our
expression (\ref{c_nk_k>sqrt}) for $c_{n,k}$ belonging to this
interval
 we have
\[
\begin{split}
K_3&=\sum_{\sqrt{n}\leqslant k\leqslant n-1}c_{n,k}k(\log
k)^{\varepsilon}
\\
&\ll \sum_{m\leqslant \sqrt{n}}
\frac{n}{m}\left(\log\frac{n}{m}\right)^\varepsilon\int_0^\infty
\frac{f_t(m)}{[n/m]^t}\,dt+\sum_{\sqrt{n}\leqslant k\leqslant
n-1}\frac{n}{k^2\log^2 n} k(\log k)^\varepsilon
\\
&\ll {n}(\log n)^\varepsilon\int_0^\infty
\frac{1}{n^{t/2}}\sum_{m\leqslant
\sqrt{n}}\frac{f_t(m)}{m}\,dt+n(\log n)^{\varepsilon-1}
\\
&\ll {n}(\log n)^\varepsilon\int_0^\infty \frac{1+t\log
n}{n^{t/2}}\,dt+n(\log n)^{\varepsilon-1}\ll n(\log
n)^{\varepsilon-1}
\end{split}
\]
Here we have used the upper bound for sum $\sum_{m\leqslant
x}\frac{f_t(m)}{m}$ provided by estimate (\ref{Lestimf_m_1}) of
Lemma \ref{Lestimf_m}.

Let us now estimate $K_2$ -- the sum over interval
$n^\alpha<m<\sqrt{n}$. We have
\begin {equation}
\label{estim_sum}
\begin{split}
K_2&=\sum_{n^{\alpha}< k < \sqrt{n}}c_{n,k}k(\log k)^{\varepsilon}
\\
 &\ll (\log n)^{\varepsilon} \sum_{n^{\alpha}< k
 <\sqrt{n}}k \int_0^\infty\frac{ F_t\left( \frac{n}{k}\right)
-F_t\bigl( \frac{n}{k+1}\bigr) }{k^t} dt
\\ &\quad+
  (\log n)^{\varepsilon} \sum_{n^{\alpha}< k
 <\sqrt{n}}\frac{n}{k}\int_{0}^\infty \frac{dt}{k^t\zeta (1+t)}
 +\sum_{n^{\alpha}< k
 <\sqrt{n}}\frac{n}{k^2\log^2 n} k(\log k)^\varepsilon
\end{split}
\end{equation}
The second and the third sum in the last estimate are clearly
$O\bigl(n(\log n)^{\varepsilon -1}\bigr)$. The first sum  in the
above upper bound can be estimated as
\begin {equation}
\begin{split}
\sum_{n^{\alpha}< k <\sqrt{n}} k \int_0^\infty\frac{ F_t\left(
\frac{n}{k}\right) -F_t\bigl( \frac{n}{k+1}\bigr) }{k^t} \,dt
&\leqslant
 \int_0^\infty\frac{1}{n^{\alpha t}}\sum_{n^{\alpha}< k <\sqrt{n}} k\left(
 F_t\Bigl( \frac{n}{k}\Bigr) -F_t\Bigl(
\frac{n}{k+1}\Bigr)\right)\, dt
\\
&\leqslant
 n\int_0^\infty\frac{1}{n^{\alpha t}}\sum_{n^{\alpha}< k <\sqrt{n}}
 \frac{k}{n}
 \sum_{\frac{n}{k+1}<m\leqslant \frac{n}{k}}f_t(m)
 \, dt
\\
&\leqslant
 n\int_0^\infty\frac{1}{n^{\alpha t}}\sum_{n^{\alpha}< k <\sqrt{n}}
 \sum_{\frac{n}{k+1}<m\leqslant \frac{n}{k}}\frac{f_t(m)}{m}
 \, dt
 \\
&\leqslant
 n\int_0^\infty\frac{1}{n^{\alpha t}}
 \sum_{m=1}^n\frac{f_t(m)}{m}
 \, dt.
\end{split}
\end{equation}
We can use the upper bound for sum $\sum_{m=1}^n\frac{f_t(m)}{m}$ as
provided in  Lemma \ref{Lestimf_m} to further estimate
\begin {equation}
\begin{split}
\sum_{n^{\alpha}< k <\sqrt{n}} k \int_0^\infty\frac{ F_t\left(
\frac{n}{k}\right) -F_t\bigl( \frac{n}{k+1}\bigr) }{k^t}
\,dt 
 &\ll n \int_0^\infty\frac{1+t\log n}{n^{\alpha t}}
 \, dt\ll \frac{n}{\log n}.
 \end{split}
\end{equation}
Inserting this estimate into (\ref{estim_sum}) we get
\begin{equation}
K_2=\sum_{n^{\alpha}< k < \sqrt{n}}c_{n,k}k(\log k)^{\varepsilon}\ll
n(\log n)^{\varepsilon-1}.
\end{equation}

The case of the $K_1$, the sum over $k$ such that $k\leqslant
n^\alpha$ is more complicated. We will prove that it is also
$O\bigl(n(\log n)^{\varepsilon-1}\bigr)$. The reason of considering
separately part $k\geqslant n^\alpha$ is that when $k\leqslant
n^\alpha$ the gap between numbers $n/k$ and $n/(k+1)$ will be large
enough to apply Theorem \ref{number_of_primes_in_interval} to
estimate the quantity $F_t(n/k)-F_t(n/(k+1))$.
 We have
\begin{equation*}
-\frac{d}{ds}L_t(s)=L_t(s)\left(-\frac{\zeta'(s)}{\zeta(s)}\right)-\frac{d}{dt}L_t(s)
\end{equation*}
which means that
\[
f_t(m)\log m=\sum_{dl=m}f_t(d)\Lambda(l)+\frac{d}{dt}f_t(m).
\]
Hence for $k\leqslant \sqrt{n}$
\begin{equation}
\label{estimF}
\begin{split}
  &F_t\Bigl( \frac{n}{k}\Bigr) -F_t\Bigl(
  \frac{n}{k+1}\Bigr)=\sum_{\frac{n}{k+1}<m\leqslant\frac{n}{k}}f_t(m)
   \\
    & =\frac{1}{\log
    \frac{n}{k}}\sum_{\frac{n}{k+1}<m\leqslant\frac{n}{k}}f_t(m)\log
    m+\frac{1}{\log
    \frac{n}{k}}\sum_{\frac{n}{k+1}<m\leqslant\frac{n}{k}}f_t(m)\log
    \frac{n}{km}
    \\
    & =\frac{1}{\log
    \frac{n}{k}}\sum_{\frac{n}{k+1}<m\leqslant\frac{n}{k}}f_t(m)\log
    m+O\left(\frac{1}{k\log
    n  }  \left(F_t\Bigl( \frac{n}{k}\Bigr) -F_t\Bigl(
  \frac{n}{k+1}\Bigr)\right)\right)
  \\
    & =\frac{1}{\log
    \frac{n}{k}}\sum_{m\leqslant\frac{n}{k}}f_t(m)
    \left(\Psi\Bigl( \frac{n}{km}  \Bigr)-\Psi\Bigl( \frac{n}{(k+1)m}  \Bigr) \right)
    \\
    &\quad+O\left(\frac{1}{\log
    n  } \frac{d}{dt} \left(F_t\Bigl( \frac{n}{k}\Bigr) -F_t\Bigl(
  \frac{n}{k+1}\Bigr)\right)\right)
    +O\left(\frac{n}{k^3\log
    n  }  \right).
\end{split}
\end{equation}
 Plugging our expression (\ref{estimF}) into our formula for $c_{n,k}$
 we have
\begin {equation}
\label{K_1}
\begin{split}
K_1
&=\sum_{2\leqslant k\leqslant n^{\alpha}}k(\log k)^\varepsilon
\left| \int_0^\infty\frac{ F_t\left( \frac{n}{k}\right) -F_t\bigl(
\frac{n}{k+1}\bigr) }{k^t} dt -\frac{n}{k(k+1)}\int_{0}^\infty
\frac{dt}{k^t\zeta (1+t)} \right|+O\bigl(n(\log n)^{\varepsilon
-1}\bigr)
\\
&\ll \sum_{k\leqslant n^\alpha} k(\log k)^{\varepsilon}
\left|\int_0^\infty
 \frac{1}{\log
    \frac{n}{k}}\sum_{m\leqslant\frac{n}{k}}f_t(m)
    \left(\Psi\Bigl( \frac{n}{km}  \Bigr)-\Psi\Bigl( \frac{n}{(k+1)m}  \Bigr)
    \right)\,\frac{dt}{k^t}\right.
    \\
    &\quad\left.-\frac{n}{k(k+1)}\int_{0}^\infty \frac{dt}{k^t\zeta (1+t)}
    \right|
\\
&\quad+\frac{1}{\log
    n  }\sum_{2\leqslant k < \sqrt{n}}k(\log k)^{\varepsilon}
\int_0^\infty\frac{1}{k^t}\frac{d}{dt}\left(F_t\Bigl(
\frac{n}{k}\Bigr) -F_t\Bigl(
  \frac{n}{k+1}\Bigr)\right)\,dt
    +O\bigl(n(\log n)^{\varepsilon -1}\bigr)
\end{split}
\end{equation}
Note that $\frac{d f_t(m)}{dt}>0$ for all $t>0$. Therefore  the sum
involving derivative   $\frac{d}{dt}\left(F_t\Bigl(
\frac{n}{k}\Bigr) -F_t\Bigl(
  \frac{n}{k+1}\Bigr)\right)>0$  can be estimated by applying partial
integration
\begin{equation}
\label{estim_with_diff}
\begin{split}
\frac{1}{\log
    n  }&\sum_{2\leqslant k < \sqrt{n}}k(\log k)^{\varepsilon}
\int_0^\infty\frac{1}{k^t}\frac{d}{dt}\left(F_t\Bigl(
\frac{n}{k}\Bigr) -F_t\Bigl(
  \frac{n}{k+1}\Bigr)\right)\,dt
  \\
  &\ll \frac{1}{\log
    n  }\sum_{2\leqslant k < \sqrt{n}}k(\log k)^{\varepsilon+1}
\int_0^\infty\left(F_t\Bigl( \frac{n}{k}\Bigr) -F_t\Bigl(
  \frac{n}{k+1}\Bigr)\right)\,\frac{dt}{k^t}
  \\
  &\ll \frac{1}{\log
    n  }\sum_{1\leqslant s\leqslant \frac{\log n}{2\log 2}}
    \sum_{2^s\leqslant k < 2^{s+1}}k(\log k)^{\varepsilon+1}
\int_0^\infty\left(F_t\Bigl( \frac{n}{k}\Bigr) -F_t\Bigl(
  \frac{n}{k+1}\Bigr)\right)\,\frac{dt}{k^t}
  \\
  &\ll \frac{1}{\log
    n  }\sum_{1\leqslant s\leqslant \frac{\log n}{2\log 2}}
    2^s s^{\varepsilon+1}
\int_0^\infty\left(F_t\Bigl( \frac{n}{2^s}\Bigr) -F_t\Bigl(
  \frac{n}{2^{s+1}}\Bigr)\right)\,\frac{dt}{2^{st}}
  \end{split}
\end{equation}
Applying now   (\ref{Lestimf_m_3}) to estimate \(F_t\bigl(
\frac{n}{2^s}\bigr) -F_t\bigl(
  \frac{n}{2^{s+1}}\bigr)\) we get
\begin{equation}
\begin{split}
\frac{1}{\log
    n  }&\sum_{2\leqslant k < \sqrt{n}}k(\log k)^{\varepsilon}
\int_0^\infty\frac{1}{k^t}\frac{d}{dt}\left(F_t\Bigl(
\frac{n}{k}\Bigr) -F_t\Bigl(
  \frac{n}{k+1}\Bigr)\right)\,dt
  \\
&\ll \frac{n}{\log
    n  }\sum_{1\leqslant s\leqslant \frac{\log n}{2\log 2}}
    s^{\varepsilon+1}
\int_0^\infty\left(t+\frac{1}{\log n}\right)\,\frac{dt}{2^{st}}
\\
    &\ll n(\log n)^{\varepsilon -1}
\end{split}
\end{equation}
Applying this estimate to continue the evaluation of $K_1$ in
(\ref{K_1}) we obtain
\begin{equation}
\begin{split}
    K_1&\ll
\sum_{k\leqslant n^\alpha} k(\log k)^{\varepsilon}
\left|\int_0^\infty
 \frac{1}{\log
    \frac{n}{k}}\sum_{m\leqslant\frac{n}{k}}f_t(m)
    \Delta\left(\frac{n}{km}  , \frac{n}{(k+1)m}
    \right)\,\frac{dt}{k^t}\right|
    \\
    &\quad+
    n\sum_{k\leqslant n^\alpha}
    \frac{(\log k)^{\varepsilon}}{(k+1)}\left|\int_{0}^\infty\biggl(\frac{1}{\zeta (1+t)}-\frac{1}{\log
    \frac{n}{k}}\sum_{m\leqslant\frac{n}{k}}\frac{f_t(m)}{m} \biggr)\frac{dt}{k^t}\right|
    +O\bigl(n(\log n)^{\varepsilon -1}\bigr).
    \end{split}
\end{equation}
The last sum in the above equation can be estimated applying
estimate (\ref{Lestimf_m_2}) of Lemma \ref{Lestimf_m} as
\[
\int_{0}^\infty\biggl(\frac{1}{\zeta (1+t)}-\frac{1}{\log
    \frac{n}{k}}\sum_{m\leqslant\frac{n}{k}}\frac{f_t(m)}{m}
    \biggr)\frac{dt}{k^t}
    \ll\frac{1}{\log n\log k}.
\]
for $k\leqslant \sqrt{n}$. This gives us
\[
K_1=\sum_{2\leqslant k\leqslant n^{\alpha}}c_{n,k}k(\log
k)^{\varepsilon} \ll D +O\bigl(n(\log n)^{\varepsilon -1}\bigr).
\]
Here
\[
D=\sum_{k\leqslant n^\alpha} k(\log k)^{\varepsilon}
\left|\int_0^\infty
 \frac{1}{\log
    \frac{n}{k}}\sum_{m\leqslant\frac{n}{k}}f_t(m)
    \Delta\left(\frac{n}{km}  , \frac{n}{(k+1)m}
    \right)\,\frac{dt}{k^t}\right|\leqslant D_1+D_2
\]
where
\[
D_1=\sum_{k\leqslant n^\alpha} k(\log k)^{\varepsilon}
\left|\int_0^\infty
 \frac{1}{\log
    \frac{n}{k}}\sum_{\frac{n}{k^{1+\delta}}\leqslant m\leqslant\frac{n}{k}}f_t(m)
    \Delta\left(\frac{n}{km}  , \frac{n}{(k+1)m}
    \right)\,\frac{dt}{k^t}\right|
\]
and
\[
D_2=\sum_{k\leqslant n^\alpha} k(\log k)^{\varepsilon} \int_0^\infty
 \frac{1}{\log
    \frac{n}{k}}\sum_{m\leqslant\frac{n}{k^{1+\delta}}}f_t(m)\left|
    \Delta\left(\frac{n}{km}  , \frac{n}{(k+1)m}
    \right)\right|\,\frac{dt}{k^t},
\]
here $\delta>0$ -- fixed, such that $\alpha(1+\delta)< 1$.

 Then
\begin{equation}
\begin{split}
    D_1&\ll
\frac{1}{\log n}\sum_{k\leqslant n^\alpha} k(\log k)^{\varepsilon}
\int_0^\infty
 \sum_{\frac{n}{k^{1+\delta}}\leqslant m\leqslant\frac{n}{k}}f_t(m)
    \left(\Psi\left(\frac{n}{km}\right)-\Psi\left( \frac{n}{(k+1)m}
    \right)\right)\,\frac{dt}{k^t}
\\
&\quad+ \frac{n}{\log n}\sum_{k\leqslant n^\alpha} \frac{(\log
k)^{\varepsilon}}{k} \int_0^\infty
 \sum_{\frac{n}{k^{1+\delta}}\leqslant
 m\leqslant\frac{n}{k}}\frac{f_t(m)}{m}
\,\frac{dt}{k^t}=:J_1+J_2.
\end{split}
\end{equation}
Changing the order of summation in $J_1$ we get
\begin{equation}
\begin{split}
J_1&=\frac{1}{\log n}\sum_{k\leqslant n^\alpha} k(\log
k)^{\varepsilon} \int_0^\infty
 \sum_{\frac{n}{k^{1+\delta}}\leqslant m\leqslant\frac{n}{k}}f_t(m)
    \left(\Psi\left(\frac{n}{km}\right)-\Psi\left( \frac{n}{(k+1)m}
    \right)\right)\,\frac{dt}{k^t}
\\
&=\frac{1}{\log n}\sum_{ n^{1-\alpha(1+\delta)}\leqslant m\leqslant
n/2} \int_0^\infty f_t(m)
 \sum_{\left(\frac{n}{m}\right)^{\frac{1}{1+\delta}}\leqslant k\leqslant\frac{n}{m}}
    k(\log k)^{\varepsilon}\left(\Psi\left(\frac{n}{km}\right)-\Psi\left( \frac{n}{(k+1)m}
    \right)\right)\,\frac{dt}{k^t}
\\
&\leqslant\frac{1}{\log n}\sum_{ n^{1-\alpha(1+\delta)}\leqslant
m\leqslant n/2}\left(\log \frac{n}{m}\right)^{\varepsilon}
\int_0^\infty f_t(m)
 \sum_{
 k\leqslant\frac{n}{m}}
    k\left(\Psi\left(\frac{n}{km}\right)-\Psi\left( \frac{n}{(k+1)m}
    \right)\right)\left(\frac{m}{n}
    \right)^{\frac{t}{1+\delta}}\,dt.
    \end{split}
\end{equation}
Since for any $x\geqslant 1$
$$
\sum_{
 k=1}^\infty
    k\left(\Psi\left(\frac{x}{k}\right)-\Psi\left( \frac{x}{k+1}
    \right)\right)=
    \sum_{
 k=1}^\infty
    \Psi\left(\frac{x}{k}\right)
    \ll x\sum_{k=1}^{[x]}\frac{1}{k} \ll x\log x
$$
we get
\begin{equation}
\label{estimate_of_J_1}
\begin{split}
J_1&\ll\frac{n}{\log n}\sum_{ n^{1-\alpha(1+\delta)}\leqslant
m\leqslant n/2}\left(\log \frac{n}{m}\right)^{\varepsilon+1}
\int_0^\infty \frac{f_t(m)}{m}\left(\frac{m}{n}
\right)^{\frac{t}{1+\delta}}
 \,dt
 \\
 &\ll\frac{n}{\log n}\sum_{ 1\leqslant s\leqslant
(1-\alpha(1+\delta))\log_2n}\sum_{ \frac{n}{2^{s+1}}< m\leqslant
\frac{n}{2^s}}\left(\log \frac{n}{m}\right)^{\varepsilon+1}
\int_0^\infty \frac{f_t(m)}{m}\left(\frac{m}{n}
\right)^{\frac{t}{1+\delta}}
 \,dt
 \\
 &\ll\frac{1}{\log n}\sum_{ 1\leqslant s\leqslant
(1-\alpha(1+\delta))\log_2n}s^{\varepsilon+1}2^s \int_0^\infty\left(
F_t\left(\frac{n}{2^s} \right)-F_t\left(\frac{n}{2^{s+1}}
\right)\right)2^{-\frac{st}{1+\delta}}
 \,dt
 \\
 &\ll\frac{n}{\log n}\sum_{ 1\leqslant s\leqslant
(1-\alpha(1+\delta))\log_2n}s^{\varepsilon+1} \int_0^\infty\left(
\frac{1}{\log n}+t\right)2^{-\frac{st}{1+\delta}}
 \,dt
 \ll n(\log n)^{\varepsilon-1}.
 \end{split}
\end{equation}
We estimate $J_2$ in a similar way as $J_1$. First changing
summation we get
\begin{equation}
\begin{split}
J_2&=\frac{n}{\log n}\sum_{ n^{1-\alpha(1+\delta)}\leqslant
m\leqslant n/2} \int_0^\infty f_t(m)
 \sum_{\left(\frac{n}{m}\right)^{\frac{1}{1+\delta}}\leqslant k\leqslant\frac{n}{m}}
    \frac{(\log k)^{\varepsilon}}{k^{1+t}}\,dt
\\&\leqslant\frac{n}{\log n}\sum_{ n^{1-\alpha(1+\delta)}\leqslant
m\leqslant n/2}\left(\log \frac{n}{m}\right)^{\varepsilon+1}
\int_0^\infty \frac{f_t(m)}{m}\left(\frac{m}{n}
\right)^{\frac{t}{1+\delta}}
 \,dt.
\end{split}
\end{equation}
The last sum has already been estimated before
 while evaluating $J_1$ in (\ref{estimate_of_J_1}), thus finally we get
\[
J_2\ll n(\log n)^{\varepsilon-1}.
\]
Our estimates of $J_1$ and $J_2$ implies that
\[
D_1\ll n(\log n)^{\varepsilon-1}.
\]
Let us now turn to estimating the sum $D_2$. Let us  chose
$\delta=1/(1-\eta)$ where $\eta$ is the same as in formulation of
Theorem \ref{number_of_primes_in_interval} then
\[
\frac{n}{mk}-\frac{n}{m(k+1)}\geqslant
\left(\frac{2}{3}\right)^{1-\eta}
\left(\frac{n}{m(k+1)}\right)^\eta,
\]
for $m\leqslant \frac{n}{k^{1+\delta}}$. Additionally let us assume
that $\alpha>0$ is small enough to ensure that $\alpha(\delta+1)<1$.
Then we can make use of Theorem \ref{number_of_primes_in_interval}
to evaluate
\[
\Delta\left(\frac{n}{km}  , \frac{n}{(k+1)m}\right) \ll
\frac{n}{k^2}\frac{1}{\log  \frac{n}{km}}.
\]
Hence we obtain
\begin{equation}
D_2\ll \frac{n}{\log
   {n}}\sum_{k\leqslant n^\alpha} \frac{(\log k)^{\varepsilon}}{k} \int_0^\infty
 \sum_{m\leqslant\frac{n}{k^{1+\delta}}}
   \frac{f_t(m)}{m\log \frac{n}{mk}}\,\frac{dt}{k^t}
\end{equation}
Changing the order of summation in the sum occurring in the last expression we get
\begin{equation}
\begin{split}
D_2&\ll \frac{n}{\log
   n}\sum_{1\leqslant m\leqslant n/2} \left(\log \frac{n}{m} \right)^{\varepsilon -1} \int_0^\infty
 \frac{f_t(m)}{m}\sum_{2\leqslant k\leqslant\left(\frac{n}{m}\right)^{\frac{1}{1+\delta}}}
   \frac{1}{k^{t+1}}\,dt
   \\
&\ll \frac{n}{\log
   n}\sum_{1\leqslant s\leqslant \log_2\frac{n}{2}}
   \sum_{\frac{n}{2^{s+1}}< m\leqslant \frac{n}{2^{s}}} \left(\log \frac{n}{m} \right)^{\varepsilon -1} \int_0^\infty
 \frac{f_t(m)}{m}\sum_{2\leqslant k\leqslant\left(\frac{n}{m}\right)^{\frac{1}{1+\delta}}}
   \frac{1}{k^{t+1}}\,dt
 \\
&\ll \frac{1}{\log
   n}\sum_{1\leqslant s\leqslant \log_2\frac{n}{2}}s^{\varepsilon
   -1}2^s
   \int_0^\infty \left( F_t\left( \frac{n}{2^{s}} \right)-F_t\left( \frac{n}{2^{s+1}}
   \right)\right)
\sum_{2\leqslant k\leqslant 2^{s+1}}
   \frac{1}{k^{t+1}}\,dt
\end{split}
\end{equation}
Applying here the estimate (\ref{Lestimf_m_3}) for $F_t(x)-F_t(x/2)$ with $x=n2^{-s}$
we further estimate
\begin{equation}
\begin{split}
D_2&\ll \frac{n}{\log
   n}\sum_{1\leqslant s\leqslant \log_2\frac{n}{2}}s^{\varepsilon
   -1}
   \int_0^\infty \left( \frac{1}{\log \frac{n}{2^s}}+t\right)
\sum_{2\leqslant k\leqslant 2^{s+1}}
   \frac{1}{k^{t+1}}\,dt
\\
&\ll \frac{n}{\log
   n}\sum_{1\leqslant s\leqslant \log_2\frac{n}{2}}s^{\varepsilon
   -1}
   \left( \frac{1}{\log \frac{n}{2^s}}\sum_{2\leqslant k\leqslant 2^{s+1}}
   \frac{1}{k\log k}+\sum_{2\leqslant k\leqslant 2^{s+1}}
   \frac{1}{k(\log k)^2}\right)
\\
&\ll n(\log n)^{\varepsilon -1}.
\end{split}
\end{equation}
Thus we have proved that
\[
K_1=\sum_{k \leqslant n^\alpha}c_{n,k}k(\log k)^{\varepsilon}\ll
D_1+D_2+n(\log n)^{\varepsilon-1}\ll n(\log n)^{\varepsilon-1}.
\]
Our estimates of $K_1$, $K_2$ and $K_3$ allow us to evaluate the sum
(\ref{estimate_of_K}) as $O\bigl(n(\log n)^{\varepsilon -1}\bigr)$,
which finally completes the proof of the theorem.
\end{proof}

\begin{proof}[Proof of Theorem \ref{first_thm}]
Plugging $S(n)=o(n\log n)$ into inequality (\ref{main_inequality})
of Theorem \ref{Tmain} and making use of properties (\ref{cond_c_m})
and (\ref{small_k}) of quantities $c_{n,k}$ we conclude that the
right hand side of (\ref{main_inequality}) is $o(n)$. Dividing both
sides of thus obtained inequality by $n$ we complete the proof of
the Theorem.
\end{proof}
\begin{lem}
\label{sufficiency} Suppose $a_k$ is a sequence of complex numbers
such that
\[
A(n)=\sum_{k\leqslant n}a_k\left[\frac{n}{k} \right]=Cn+o(n),
\]
as $n\to \infty$, with some constant $C\in \mathbb{C}$.

Then
\[
S(n)=\sum_{k\leqslant n}a_k\left[\frac{n}{k} \right]\log k=o(n\log n
),
\]
as $n\to \infty$.
\end{lem}
\begin{proof}
The equality of $f(m)$ to the sum $\sum_{d|m}a_d$ is equivalent to
identity
\[
U(s)=\sum_{m=1}^\infty\frac{f(m)}{m^s}=\sum_{m=1}^\infty
\frac{1}{m^s}\sum_{m=1}^\infty \frac{a_m}{m^s}=\zeta(s)g(s).
\]
 Therefore
\[
\zeta(s)g'(s)=(\zeta(s)g(s))'-\zeta'(s)g(s)=U'(s)-\frac{\zeta'(s)}{\zeta(s)}U(s)
\]
this identity corresponds to the equality of the coefficients of
$m^{-s}$ of the corresponding Dirichlet series
\[
\sum_{k|m}a_k\log k=f(m)\log m -\sum_{k\ell=m}\Lambda(k)f(\ell),
\]
for all $m\geqslant 1$. Summing the above identity over all $m$ such
that $m\leqslant n$ and recalling that $f(1)+f(2)+\cdots+f(k)=A(k)$
we get
\begin {equation}
\begin{split}
S(n)&= \sum_{k=1}^nf(k)\log k-\sum_{k\ell\leqslant
n}\Lambda(k)f(\ell)\\
& =\sum_{k=1}^n\bigl(A(k)-A(k-1)\bigr)\log k-\sum_{k\leqslant
n}\Lambda(k)A\left(\frac{n}{k} \right)\\
& =A(n)\log n-\sum_{k=1}^nA(k)\log \left(1+\frac{1}{k}
\right)-\sum_{k\leqslant n}\Lambda(k)A\left(\frac{n}{k} \right).
\end{split}
\end{equation}
By condition of the lemma $A(n)=Cn+o(n)$. Inserting this estimate
into the above expression of $S(n)$ we get
\begin {equation}
\begin{split}
S(n) & =A(n)\log n-\sum_{k\leqslant n}\Lambda(k)A\left(\frac{n}{k}
\right)+O(n)\\
& =Cn\log n-Cn\sum_{k\leqslant n}\frac{\Lambda(k)}{k}+o(n\log
n)=o(n\log n),
\end{split}
\end{equation}
where we have used the fact that $\sum_{k\leqslant
n}\frac{\Lambda(k)}{k}=\log n +O(n)$.

  The lemma is proved.
\end{proof}

\begin{proof}[Proof of theorem \ref{Ttauberian_dirichlet}]
The sufficiency of the two conditions of the  theorem for the
existence of the limit of the sum (\ref{F_a_k}) follows immediately
from the Theorem \ref{first_thm}.

The necessity of the first condition of the theorem follows from
Lemma \ref{sufficiency}. The necessity of the second condition will
follow if we note that function $g(s)$ can be represented as a
fraction
\[
g(s)=\frac{\zeta(s)g(s)}{\zeta(s)}=\frac{\sum_{m=1}^\infty\frac{f(m)}{m^s}}{\sum_{m=1}^\infty\frac{1}{m^s}},
\]
where as before $f(m)=\sum_{d|m}a_d$. The partial sums of the
coefficients of the Dirichlet series in the nominator satisfies
$f(1)+f(2)+\cdots+f(n)=\sum_{k=1}^na_k\left[\frac{n}{k}\right]=Cn+o(n)$,
by our assumtion. Thus passing to the limit $s\downarrow 1$ we
conclude that $\lim_{s \downarrow 1} g(s)=C$.
\end{proof}
\begin{proof}[Proof of theorem \ref{T_multiplicative}]
Once we are given the values of $f$  on prime numbers $p$ such that
$p\leqslant n$ we can compute the value of function $f$ on any
integer $m$ such that $m\leqslant n$. The numbers $f(p)$, with $p>n$
do not influence the value of the quantity
\[
\frac{1}{n}\sum_{m=1}^nf(m),
\]
therefore we will assume that $f(p)=1$ for $p>n$.

We have already noted that if $f(m)=\sum_{d|m}a_d$ then the
Dirichlet generating function $U(s)$ of $f(m)$ can be represented as
a product
\[
U(s)=\sum_{m=1}^\infty\frac{f(m)}{m^s}=\zeta(s)g(s).
\]
On the other hand  by the condition of the theorem \(f(m)\) is a
multiplicative function, which means that its generating function
can be represented as Euler product
\[
U(s)=\sum_{m=1}^\infty\frac{f(m)}{m^s}=\prod_{p}\left(1-\frac{f(p)}{p^s}
\right)^{-1}.
\]
Comparing the above two expressions of $U(s)$ we conclude that that
the Dirichlet series of numbers $a_m$ such that $f(m)=\sum_{d|m}a_d$
is
\[
g(s)=\sum_{m=0}^\infty\frac{a_m}{m^s}=\frac{1}{\zeta(s)}\prod_{p}\left(1-\frac{f(p)}{p^s}
\right)^{-1}=\prod_{p}\frac{1-\frac{1}{p^s} }{1-\frac{f(p)}{p^s}
}=\exp\Biggl\{\sum_{p}\sum_{k\geqslant 1}
\frac{f(p^k)-1}{kp^{ks}}\Biggr\}.
\]
Differentiating this expression of $g(s)$ we obtain that this
function satisfies differential equation
$g'(s)=-g(s)\sum_{m=1}^\infty \frac{f(m)-1}{m^s}\Lambda(m)$.
Multiplying both sides of this equation by $\zeta(s)$ and using the
fact that  the fact that $U(s)=\zeta(s)g(s)$ we obtain an identity
\[
\zeta(s)g'(s)=-U(s)\sum_{m=1}^\infty \frac{f(m)-1}{m^s}\Lambda(m).
\]
or equivalently
\[
\sum_{m=1}^\infty \frac{1}{m^s}\sum_{k=1}^\infty \frac{a_k\log
k}{k^s}=\sum_{k=1}^\infty\frac{f(k)}{k^s}\sum_{m=1}^\infty
\frac{f(m)-1}{m^s}\Lambda(m).
\]
Equating the coefficients of $d^{-s}$ in the  Dirichlet series on
both sides of the above identity and summing over all $d$ such that
$d\leqslant m$ we obtain
\[
\begin{split}
S(m)&=\sum_{d\leqslant m}\sum_{k|d}a_k\log k=\sum_{d\leqslant m}
\sum_{k|d}\bigl(f(k)-1\bigr)\Lambda(k)f\left(\frac{d}{k}\right)
\\
&=\sum_{k\leqslant m}
\bigl(f(k)-1\bigr)\Lambda(k)\sum_{\ell\leqslant m/k}f(\ell )
\end{split}
\]
 Therefore,
recalling that according to the condition of the theorem
$|f(m)|\leqslant 1$, and recalling that $f(p)=1$ for $p>n$, we can
estimate
\begin{equation}
\begin{split}
|S(m)|&\leqslant \sum_{k\leqslant m}
\bigl|f(k)-1\bigr|\Lambda(k)\Bigl|\sum_{\ell\leqslant
m/k}f(\ell)\Bigr|=\sum_{k\leqslant
m}\bigl|f(k)-1\bigr|\left[\frac{m}{k}\right]\Lambda(k)
\\
&
\ll m \sum_{p\leqslant m}\frac{\bigl|f(p)-1\bigr|}{p}\log p \ll
m(\log m)^{1/\beta} \left(\sum_{p\leqslant
n}\frac{\bigl|f(p)-1\bigr|^\alpha}{p}\log p\right)^{1/\alpha}
\\
&\ll   m(\log m)^{1/\beta}(\log n)^{1/\alpha}\mu_n(\alpha),
\end{split}
\end{equation}
for \(m\geqslant 2\). Here we have applied the Cauchy inequality
with parameters \(\frac{1}{\alpha}+\frac{1}{\beta}=1\). Inserting
this estimate of \(S(n)\) into the inequality of Theorem \ref{Tmain}
with \(\varepsilon=\frac{1}{\beta}\) we get
\[
\frac{1}{n}\sum_{m=1}^nf(m)=g\left(1+\frac{1}{\log
n}\right)+O(\mu_n(\alpha)),
\]
An easy calculation yields
\[
g\left(1+\frac{1}{\log n}\right)=\prod_{p\leqslant
n}\frac{1-\frac{1}{p^{1+{1}/{\log n}}}
}{1-\frac{f(p)}{p^{1+{1}/{\log n}}} }=\prod_{p\leqslant
n}\frac{1-\frac{1}{p} }{1-\frac{f(p)}{p}
}\bigl(1+O(\mu_n(\alpha))\bigr).
\]
Hence follows the proof of the theorem.
\end{proof}

\bibliographystyle{plain}
\def\polhk#1{\setbox0=\hbox{#1}{\ooalign{\hidewidth
  \lower1.5ex\hbox{`}\hidewidth\crcr\unhbox0}}}
  \def\polhk#1{\setbox0=\hbox{#1}{\ooalign{\hidewidth
  \lower1.5ex\hbox{`}\hidewidth\crcr\unhbox0}}}
  \def\polhk#1{\setbox0=\hbox{#1}{\ooalign{\hidewidth
  \lower1.5ex\hbox{`}\hidewidth\crcr\unhbox0}}}
  \def\polhk#1{\setbox0=\hbox{#1}{\ooalign{\hidewidth
  \lower1.5ex\hbox{`}\hidewidth\crcr\unhbox0}}}
  \def\polhk#1{\setbox0=\hbox{#1}{\ooalign{\hidewidth
  \lower1.5ex\hbox{`}\hidewidth\crcr\unhbox0}}}
  \def\polhk#1{\setbox0=\hbox{#1}{\ooalign{\hidewidth
  \lower1.5ex\hbox{`}\hidewidth\crcr\unhbox0}}}

\end{document}